\documentclass[12pt]{amsart}
\usepackage{amssymb}
\usepackage{amsmath}
\usepackage{amsthm}
\newtheorem{theorem}{Theorem}[section]
\newtheorem{proposition}{Proposition}[section]
\newtheorem{corollary}[theorem]{Corollary}
\newtheorem{lemma}[theorem]{Lemma}

\newcommand{\e}{\varepsilon}
\newcommand{\F}{{\mathcal F}}

\newcommand{\U}{{\mathcal U}}

\DeclareMathOperator{\Ker}{Ker}
\DeclareMathOperator{\End}{End}
\DeclareMathOperator{\tr}{tr}

\begin{document}

\title[A mapping from unitary to doubly stochastic matrices]{A mapping from the unitary to doubly stochastic matrices and symbols on a finite set}

\author[A.V. Karabegov]{Alexander V. Karabegov}
\address[Alexander V. Karabegov]{Department of Mathematics, Abilene Christian University, ACU Box 28012, 215 Foster Science Building, Abilene, TX 79699-8012}
\email{axk02d@acu.edu}
\begin{abstract}
  We prove that the mapping from the unitary to the doubly stochastic matrices
that maps a unitary matrix $(u_{kl})$ to the doubly stochastic matrix $(|u_{kl}|^2)$ is a submersion for almost all unitary matrices. The proof uses the framework of operator symbols on a finite set.
\end{abstract}
\subjclass{Primary: 53A07; Secondary: 35S05}
\keywords{doubly stochastic matrices, symbols}

%\date{February 20, 2001, rev. June 18, 2001}
\maketitle

\section{Introduction}

Given a unitary matrix $(u_{kl})$, the matrix  $(|u_{kl}|^2)$ is doubly stochastic. The mapping 
\[
   \mu: (u_{kl}) \mapsto (|u_{kl}|^2)
\]
from the set $U$ of unitary $(n\times n)$-matrices to the set $P$ of doubly stochastic $(n\times n)$-matrices was studied in a number of papers in mathematical physics (see \cite{TZ} and references therein). The set $U$ is a real $n^2$-dimensional manifold and $P$ is a convex $(n-1)^2$-dimensional polytope, the Birkhoff polytope. It is known that the image $\mu(U)$ is a proper subset of the Birkhoff polytope $P$. Two unitary matrices $u$ and $\tilde u$ are called equivalent if there exist unitary diagonal matrices $\varkappa$ and $\lambda$ such that $\tilde u = \varkappa u \lambda$. The mapping $\mu$ maps equivalent unitary matrices to  the same doubly stochastic matrix. The set of equivalence classes of unitary matrices with nonzero entries is an $(n-1)^2$-dimensional manifold. These dimensional considerations lead to the natural question of whether $\mu$ is a submersion. It was deemed that $\mu$ has to be a submersion almost everywhere on $U$ (see, e.g., \cite{B}). In this paper we give a detailed proof of this result which was announced in \cite{K}.

To each unitary $(n\times n)$-matrix $u$ with nonzero entries we relate an operator $I_u$ on an $n^2$-dimensional vector space $V$ and show that the dimension of the kernel of the tangent mapping $\mu_{*u}$ equals the multiplicity of the eigenvalue 1 in the spectrum of $I_u$. We introduce the notion of a symmetry group $G$ of a unitary matrix $u$. The group $G$ has a unitary representation in the space $V$ which commutes with the operator $I_u$. When the matrix $u$ has a large symmetry group, it is sometimes possible to calculate the spectral decomposition of the operator $I_u$. This allows us to show that the mapping $\mu$ is a submersion for almost all elements of $U$.

\section{Symbols on a finite set}\label{S:symbols}

In this section we introduce a construction of two types of symbols of operators on a finite dimensional vector space. These symbols are analogues of $pq$- and $qp$-symbols in quantum mechanics (see \cite{BS}).

We will model finite dimensional vector spaces as the spaces of functions on finite sets. Given a finite set $J$, denote by $F(J)$ the space of all complex-valued functions on $J$. The standard Hermitian product of functions $\varphi = \varphi_j, \psi = \psi_j \in F(J)$
is given by the formula 
\[
    \langle \varphi,\psi\rangle = \sum_{j\in J} \varphi_j \bar \psi_j.
\]
Now let $K$ and $L$ be two $n$-element index sets and $u = (u_{kl}), k\in K,l\in L$, be a unitary matrix. It determines a unitary isomorphism $\U : F(L) \to F(K)$ of the vector spaces $F(K)$ and $F(L)$ endowed with the standard Hermitian products: for $\psi = \psi_l \in F(L)$,
\[
   (\U \psi)_k = \sum_{l\in L} u_{kl} \psi_l.  
\]
Denote $M := K\times L$. We define two mappings
\[ 
C_u,D_u : F(M) \to \End F(K)
\]
such that for $f = f_{kl} \in F(M)$ the matrices $(x_{kk'})$ and $(y_{kk'})$ of the operators $C_uf$ and $D_uf$ in $F(K)$, respectively, are given by the following formulas:
\begin{equation}\label{E:cdsymb}
   x_{kk'} = \sum_{l \in L} u_{kl} f_{kl} \bar u_{k'l} \mbox{ and } y_{kk'} = \sum_{l \in L} u_{kl} f_{k'l} \bar u_{k'l}.
\end{equation}
The mappings $f \mapsto C_uf$ and $f \mapsto D_uf$ can be thought of as two different symbol-to-operator mappings. 

Given functions $a=a_k \in F(K)$ and $b = b_l\in F(L)$, denote by $\hat a$ and $\hat b$ the multiplication operators by these functions in $F(K)$ and $F(L)$, respectively. Then $ab = a_kb_l \in F(M), \ \U\hat b \U^* \in \End F(K),$ and it can be checked that
\begin{equation}\label{E:ord}
    C_u(ab) = \hat a \left(\U \hat b \U^*\right) \mbox{ and } D_u(ab) = \left(\U \hat b \U^*\right) \hat a. 
\end{equation}
Using the analogy with quantum mechanics, we may think of $F(K)$ and $F(L)$ as of the ``coordinate" and``momentum" representations of an abstract Hermitian vector space, respectively, with the isomorphism $\U$ playing the role of the Fourier transform. The mapping $C_u$ is obtained by the standard ordering of the ``coordinate" operators $\hat a$ and ``momentum" operators $\U\hat b \U^*$, while $D_u$ is obtained via the inverse ordering.

Given a function $f \in F(M)$ and an operator $A \in \End F(K)$, we will say that the function $f$ is a $C$-symbol of the operator $A$ if $A = C_u f$ and that $f$ is a $D$-symbol of $A$ if $A = D_u f$. Thus, $C$- and $D$-symbols correspond to $qp$- and $pq$-symbols, respectively. 

Now we will make a critical assumption that all elements $u_{kl}$ of the matrix $u$ are nonzero and introduce a Hermitian product $\langle\cdot,\cdot\rangle_u$ on the space $F(M)$ by the formula
\[
     \langle f,g\rangle_u = \sum_{k\in K, l\in L} f_{kl} \bar g_{kl}|u_{kl}|^2.
\]
The space $\End F(K)$ of operators in $F(K)$ carries the Hilbert-Schmidt Hermitian product $\langle\cdot,\cdot\rangle_{HS}$. If $(x_{kk'})$ and $(y_{kk'})$ are the matrices of operators $X$ and $Y$ in $F(K)$, respectively, then 
\[
    \langle X,Y \rangle_{HS} = \tr XY^\ast = \sum_{k, k' \in K} x_{kk'}\bar y_{kk'}.
\]
The following properties of $C$- and $D$-symbols are analogous to the corresponding properties of $pq$- and $qp$-symbols of Hilbert-Schmidt operators in quantum mechanics.
\begin{proposition}\label{P:isom}
   The mappings $C_u, D_u: F(M) \to \End F(K)$ are unitary isomorphisms of the Hermitian vector spaces $(F(M),\langle\cdot,\cdot\rangle_u)$ and $(\End V, \langle\cdot,\cdot\rangle_{HS})$. 
\end{proposition}
\begin{proof}
 Given two functions $f,g \in F(M)$, let $(x_{kk'})$ and $(y_{kk'})$ be the matrices of the operators $C_uf$ and $C_ug$, respectively, so that
\[
   x_{kk'} = \sum_{l \in L} u_{kl} f_{kl} \bar u_{k'l} \mbox{ and } y_{kk'} = \sum_{l' \in L} u_{kl'} g_{kl'} \bar u_{k'l'}.
\]
Then, using the unitarity of the matrix $u$, we obtain that
\begin{eqnarray*}
   \langle C_uf, C_ug \rangle_{HS} = \sum_{k,k'\in K} x_{kk'} \bar y_{kk'} = 
\sum_{k,k'\in K, l,l'\in L} u_{kl} f_{kl} \bar u_{k'l} \bar u_{kl'} \bar g_{kl'} u_{k'l'}=\\
\sum_{k\in K, l,l'\in L} u_{kl} f_{kl} \bar u_{kl'} \bar g_{kl'} \delta_{ll'}=
\sum_{k\in K,l\in L} f_{kl}\bar g_{kl} |u_{kl}|^2 = \langle f,g\rangle_u,
\end{eqnarray*}
where $\delta_{ll'}$ is the Kronecker symbol. Thus $C_u: F(M) \to \End F(K)$ is an isometric mapping. Since the dimensions of $F(M)$ and $\End F(K)$ are both equal to $n^2$, the mapping $C_u$ is a unitary isomorphism. A similar calculation shows that $D_u: F(M) \to \End F(K)$ is also a unitary isomorphism.
\end{proof}

It follows from Proposition \ref{P:isom} that the mapping $I_u := C_u^{-1}D_u$ that maps the $D$-symbols to the corresponding $C$-symbols is a unitary operator in $(F(M),\langle\cdot,\cdot\rangle_u)$. Using the analogy with quantum mechanics, we will call the operator $I_u$ the Berezin transform. It can be checked that the operator $I_u$ is given by the following formula: for a function $f = f_{kl} \in F(M)$,
\[
     (I_u f)_{kl} = \sum_{k' \in K, l'\in L} \frac{u_{kl'} u_{k'l}}{u_{kl} u_{k'l'}}\, f_{k'l'}\, |u_{k'l'}|^2.
\]
\begin{proposition}\label{P:conjug} Given a function $f \in F(M)$, the operators $C_uf$ and $D_u\bar f$ are Hermitian conjugate. 
\end{proposition}
\begin{proof}
  Let $(x_{kk'})$ be the matrix of the operator $C_uf$ in $F(K)$. Using formulas (\ref{E:cdsymb}) we see that the Hermitian conjugate matrix $(\bar x_{k'k})$ is given by the equation
\[
    \bar x_{k'k} = \sum_{l\in L} \bar u_{k'l} \bar f_{k'l} u_{kl}
\]
and therefore coincides with the matrix of the operator $D_u\bar f$. 
\end{proof}

This proposition allows to give a characterization of the $C$- and $D$-symbols of the skew-Hermitian operators in $F(K)$ which will be used in Section \ref{S:geom}.

\begin{corollary}\label{C:skew}
  Given functions $f,g \in F(M)$, the operator $C_uf$ is skew-Hermitian if and only if $f = - I_u \bar f$ and the operator $D_ug$ is skew-Hermitian if and only if $I_u g = -\bar g$. If $f = I_u g$, then the operator $C_u f = D_u g$ is skew-Hermitian if and only if $f = - \bar g$.
\end{corollary}
\begin{proof} Proposition \ref{P:conjug} states that
\[
    (C_uf)^*= D_u\bar f.
\]
Now the operator $C_uf$ is skew-Hermitian, i.e., $C_uf = -(C_uf)^*$, if and only if $C_u f = - D_u \bar f$ or, equivalently, $f = -I_u\bar f$. The rest of the corollary is proved similarly. 
\end{proof}

Assume that functions $f, g\in F(M)$ are such that $f = I_ug$ or, equivalently, $C_uf = D_ug$. According to Proposition \ref{P:conjug}, 
\[
D_u\bar f = (C_uf)^* = (D_ug)^* = C_u\bar g,
\]
whence $I_u\bar f = \bar g$. Let $g\in F(M)$ be an eigenfunction of the Berezin transform $I_u$ with the eigenvalue $\theta$. Since $I_u$ is unitary, we have $|\theta| =1$. Equation $I_u g = \theta g$ implies that $I_u(\bar\theta \bar g) = \bar g$, whence $I_u(\bar g) = \theta \bar g$. Thus we have proved the following 
\begin{lemma}\label{L:eigen}
The eigenspaces of the Berezin transform $I_u$ are invariant under complex conjugation. In particular, for any eigenvalue $\theta$ of multiplicity $k$ the $\theta$-eigenspace treated as a real vector space is a direct sum of the $k$-dimensional real vector spaces comprised of real and purely imaginary eigenfunctions, respectively. 
\end{lemma}
Denote by $E$ the subspace of $F(M)$ of the functions $f$ that can be represented as a sum of the form $f_{kl} = a_k + b_l$. The dimension of $E$ is $2n - 1$.
Formulas (\ref{E:ord}) imply that for functions $a = a_k$ and $b = b_l$ considered as elements of $F(M)$,
\[
     C_u(a) = D_u(a) \mbox{ and } C_u(b) = D_u(b),
\] 
which leads to the following statement.
\begin{lemma}\label{L:E}
 For any function $f \in E$ we have $I_uf = f$. Thus the multiplicity of 1 in the spectrum of the Berezin transform $I_u$ is at least $2n-1$.
\end{lemma}
 
\section{Symmetries of a unitary matrix}

Assume that, as in Section \ref{S:symbols}, $K$ and $L$ are two $n$-element sets, and $u = (u_{kl}),\, k \in K,l\in L,$ is a unitary matrix which defines a unitary isomorphism
\[
    \U: F(L) \to F(K),
\]
where $F(K)$ and $F(L)$ are endowed with the standard Hermitian products. Assume further that $G$ is a group that acts on the left on the sets $K$ and $L$ and has faithful unitary representations $S$ and $T$ in the spaces $F(K)$ and $F(L)$, respectively, such that, given $g\in G, \varphi=\varphi_k \in F(K), \psi = \psi_l \in F(L)$,
\[
     \left(S_g \varphi\right)_k = a_k(g) \varphi_{g^{-1}k} \mbox{ and } \left(T_g \psi\right)_l = b_l(g) \psi_{g^{-1}l},
\]
where $a_k(g) \in F(K)$ and $b_l(g) \in F(L)$ are two unitary functions (i.e., $|a_k(g)| = |b_l(g)| = 1$). Finally we assume that the isomorphism $\U$ intertwines the representations $S$ and $T$: for $g \in G$,
\begin{equation}\label{E:inter}
                 S_g\,  \U = \U\,  T_g.
\end{equation}
If these conditions are satisfied, we say that $G$ is a {\it symmetry group of the matrix} $u$. Condition (\ref{E:inter}) can be expressed in terms of the matrix $u$ as follows:
\begin{equation}\label{E:matrixinter}
     a_k(g)\left( u_{g^{-1}k\, g^{-1}l} \right) \bar b_l(g) = u_{kl}.
\end{equation}
Denote by $R$ the following representation of the group $G$ in $F(M)$: for $g \in G$ and $f = f_{kl} \in F(M)$,
\[
       \left(R_g f\right)_{kl} = f_{g^{-1}k\, g^{-1}l}.
\]
Using formula (\ref{E:matrixinter}) and the fact that the functions $a_k(g)$ and $b_l(g)$ are unitary we obtain the following 
\begin{lemma}\label{L:unitary}
  The representation $R$ is a unitary representation of the group $G$ in the Hermitian vector space $(F(M), \langle\cdot,\cdot\rangle_u)$.
\end{lemma}

Denote by $Q$ the representation of the group $G$ in the space $\End F(K)$ by conjugation via the operators of the representation $S$: for $g \in G$ and $X \in \End F(K)$,
\[
     Q_g(X) = S_g X S_g^{-1} = S_g X S_g^*.
\]
The representation $Q$ is a unitary representation of the group $G$ in the Hermitian vector space $(\End F(K),\langle\cdot,\cdot\rangle_{HS})$.

The following equivariance property of the symbol-to-operator mappings $C_u$ and $D_u$ is analogous to the equivariance of the $pq$- and $qp$-symbol mappings with respect to the Heisenberg group.

\begin{proposition}\label{P:equiv}
The mappings $C_u, D_u: F(M) \to \End F(K)$ intertwine the representations $Q$ and $R$: for $g \in G$,
\[
     C_u(R_gf) = S_g(C_u(f))S_g^* \mbox{ and } D_u(R_gf) = S_g(D_u(f))S_g^*.
\]
\end{proposition}
\begin{proof}
  Given $g \in G, \ f \in F(M)$ and $\varphi \in F(K)$, we get with the use of formula (\ref{E:matrixinter}) the following:
\begin{eqnarray*}
 (S C_u(f) \varphi)_k = \sum_{k' \in K, l'\in L} a_k(g) u_{g^{-1}k \,l'} f_{g^{-1}k \,l'} \bar u_{k'l'} \varphi_{k'} = \\
\sum_{\tilde k \in K, l\in L} a_k(g) u_{g^{-1}k \,g^{-1}l} f_{g^{-1}k \,g^{-1}l} \bar u_{g^{-1}\tilde k \,g^{-1}l} \varphi_{g^{-1}\tilde k} = \\
\sum_{\tilde k \in K, l\in L} a_k(g) u_{g^{-1}k \,g^{-1}l}\bar b_l(g) f_{g^{-1}k \,g^{-1}l}b_l(g) \bar u_{g^{-1}\tilde k \,g^{-1}l} \varphi_{g^{-1}\tilde k} = \\
\sum_{\tilde k \in K, l\in L} u_{kl}(R_g f)_{kl}b_l(g) \bar u_{g^{-1}\tilde k \,g^{-1}l} \varphi_{g^{-1}\tilde k} = \\
\sum_{\tilde k \in K, l\in L} u_{kl}(R_gf)_{kl} \left(\overline{ a_{\tilde k}(g)u_{g^{-1}\tilde k\, g^{-1}l}\bar b_l(g)}\right) a_{\tilde k}(g) \varphi_{\sigma(\tilde k)} = \\
\sum_{\tilde k \in K, l\in L} u_{kl}(R_g f)_{kl} \bar u_{\tilde kl}a_{\tilde k}(g) \varphi_{g^{-1}\tilde k} = \left(C_u\left(R_gf\right) S \varphi\right)_k.
\end{eqnarray*}
The proof that the mapping $D_u$ intertwines the representations $Q$ and $R$ is similar.
\end{proof}
Since both $C_u$ and $D_u$ intertwine the representations $Q$ and $R$ of the group $G$, we obtain the following
\begin{corollary}
 The Berezin transform $I_u = C_u^{-1} D_u$ commutes with the representation $R$: for $g\in G$,
\[
               I_u R_g = R_g I_u.
\]
\end{corollary}

\section{Examples of unitary matrices with a large group of symmetries}\label{S:examples}

In this section we will consider two examples of unitary matrices with large symmetry groups. For each of these matrices we will obtain the spectral decomposition of the corresponding Berezin transform and in particular determine the multiplicity of the eigenvalue 1 in their spectra. In Section \ref{S:geom} we will give a geometric interpretation of this multiplicity.  

{\it Example 1.} Fix a natural number $n$ and denote by $\e$ the character of the cyclic group $\mathbb{Z}/n\mathbb{Z}$ given by the formula
\[
     \e(k) = \exp\left\{\frac{2\pi i k}{n}\right\}.
\]
Here we denote by $k$ (somewhat loosely) an integer modulo $n$. 
The Fourier transform $\F$ on the cyclic group $\mathbb{Z}/n\mathbb{Z}$ is given by the unitary matrix $u = (u_{kl})$ with
\[
     u_{kl} = \frac{1}{\sqrt{n}}\,  \e(kl),
\] 
where $k,l \in \mathbb{Z}/n\mathbb{Z}$ (here $K = L = \mathbb{Z}/n\mathbb{Z}$).
The corresponding Berezin transform is given by the following formula:
\begin{equation}\label{E:fourber}
    (I_uf)_{kl} = \frac{1}{n}\sum_{k',l' \in \mathbb{Z}/n\mathbb{Z}}\e(-(k - k')(l-l')) \,  f_{k'l'}. 
\end{equation}
For $r \in \mathbb{Z}/n\mathbb{Z}$ denote by $W_r$ the unitary operator of multiplication by the function $\e(rk)$ and by $Z_r$ the shift operator by the element $r$ such that for $\varphi_k \in F(\mathbb{Z}/n\mathbb{Z})$
\[
     (W_r \varphi)_k = \e(rk) \varphi_k \mbox{ and } (Z_r \varphi)_k = \varphi_{k + r}.
\]
 A simple calculation shows that 
\begin{equation}\label{E:fmt}
    \F^* W_r \F = Z_{-r} \mbox{ and } \F^* Z_r \F= W_r.
\end{equation}
Formulas (\ref{E:ord}) and (\ref{E:fmt}) imply that
\[
    C_u(\e(rk + sl)) = W_r \F W_s \F^* = W_r Z_s.
\]
Similarly,
\[ 
    D_u(\e(rk + sl)) = Z_s W_r.
\]
From the Weyl commutation relations
\begin{equation}\label{E:Weyl}
     Z_s W_r = \e(rs) W_r Z_s
\end{equation}
it follows immediately that
\[
    I_u(\e(rk + sl)) = \e(rs)\e(rk + sl).
\]
The functions $\e(rk + sl)$ with $r,s \in \mathbb{Z}/n\mathbb{Z}$ form a complete system of eigenvectors of the operator $I_u$ with the corresponding eigenvalues $\e(rs)$. Observe that $\e(rs) =1$ if and only if $rs = 0$ in $\mathbb{Z}/n\mathbb{Z}$, which implies the following
\begin{lemma}
 The multiplicity of the eigenvalue 1 of the Berezin transform corresponding to the matrix of the Fourier transform on the group $\mathbb{Z}/n\mathbb{Z}$ is equal to the number of pairs $(r,s)$ of integers such that $0 \leq r,s < n$ and $rs = 0 \pmod n$. In particular, this multiplicity equals $2n-1$ if and only if $n$ is a prime.
\end{lemma}

{\it Remark.}  {\small The spectral decomposition of the Berezin transform $I_u$ is easy to obtain because the matrix $u$ of the Fourier transform on the group $\mathbb{Z}/n\mathbb{Z}$ has a large symmetry group. Let $\mathbb{T}$ denote the multiplicative group of complex numbers of modulus 1. The generalized Heisenberg group $ G := (\mathbb{Z}/n\mathbb{Z} \times \mathbb{Z}/n\mathbb{Z}) \ltimes \mathbb{T}$ with the product
\[
    (s_1, t_1, \theta_1) \cdot (s_2, t_2, \theta_2) = (s_1 + s_2, t_1 + t_2, \theta_1\theta_2 \varepsilon(t_1 s_2)),
\]
where $s_i,t_i \in \mathbb{Z}/n\mathbb{Z}$ and $\theta_i \in \mathbb{T}$, is a symmetry group of the matrix $u$. The corresponding representations $S$ and $T$ are given by the formulas
\[
           S_{(s,t,\theta)} = \theta W_s Z_t \mbox{ and } T_{(s,t,\theta)} = \F^* S_{(s,t,\theta)} \F = \theta Z_{-s} W_t = \theta \varepsilon(-st) W_t Z_{-s}.
\]
Given functions $\varphi_k$ and $\psi_l$ on $\mathbb{Z}/n\mathbb{Z}$ and an element $(s,t,\theta) \in G$, we have
\[
   \left(S_{(s,t,\theta)}\varphi\right)_k = \theta \varepsilon(sk)\varphi_{k + t} \mbox{ and } \left(T_{(s,t,\theta)}\psi\right)_l = \theta\varepsilon((l-s)t)\psi_{l-s}.
\]
The corresponding representation $R$ acts upon the space $F(M) = F\left((\mathbb{Z}/n\mathbb{Z})^2\right)$ via shifts: for  $(s,t,\theta) \in G$ and $f = f_{kl} \in F(M)$,
\[
     \left(R_{(s,t,\theta)}f\right)_{kl} = f_{k+t \, l -s}.
\]
Since the Berezin transform (\ref{E:fourber}) commutes with the representation $R$ it is shift-invariant and thus the functions $\e(rk + sl)$ with $r,s \in \mathbb{Z}/n\mathbb{Z}$ are its eigenfunctions.}

{\it Example 2.} Let $n$ be a fixed natural number. Set $K = \{1,2,\cdots,n\}$ and denote by $\mathfrak{S}_n$ the group of permutations of $K$ (the symmetric group). Fix a constant  $\theta$ such that $|\theta|=1$ and $\theta \neq \pm 1$. A simple check shows that the matrix $u = (u_{kl}),\, k,l \in K$, with
\begin{equation}\label{E:example2}
     u_{kl} = \delta_{kl} + \frac{\theta - 1}{n}
\end{equation}
has nonzero elements and is unitary. Denote by $\U \in \End F(K)$ the corresponding unitary operator (we assume that $F(K)$ carries the standard Hermitian structure). 
Let $S$ be the unitary representation of the symmetric group $\mathfrak{S}_n$ in the space $F(K)$ by permutations: for $\sigma \in \mathfrak{S}_n$ and $\varphi = \varphi_k \in F(K)$, 
\[
       (S_\sigma \varphi)_k = \varphi_{\sigma^{-1}(k)}.
\]
Since the operators $\U$ and $S_\sigma$ commute, the group $\mathfrak{S}_n$ is a symmetry group of the matrix $u$. According to Lemma \ref{L:unitary} the representation $R$ of the group $\mathfrak{S}_n$ in the Hermitian vector space $(F(M), \langle\cdot,\cdot\rangle_u)$
given by the formula
\[
     \left(R_\sigma f\right)_{kl} = f_{\sigma^{-1}(k)\sigma^{-1}(l)}.
\]
is unitary. Here $\sigma\in \mathfrak{S}_n$  and $f = f_{kl} \in F(M)$.

The representation $R$ has the following four isotypic components. Denote by $V_1$ the two-dimensional space of functions from $F(M)$ of the form $f_{kl} = a + b\delta_{kl}$, where $a$ and $b$ are constants. The group $\mathfrak{S}_n$ acts trivially on $V_1$. The space $V_2 \subset F(M)$ of functions of the form $f_{kl} = a_k + b_l + c_k\, \delta_{kl}$ where
\[
     \sum_{i = 1}^n a_i = \sum_{i = 1}^n b_i = \sum_{i = 1}^n c_i =0,
\] 
is the isotypic component of multiplicity 3 of the representation $R$ corresponding to the irreducible representation of $\mathfrak{S}_n$ of dimension $n-1$ parametrized by the partition $(n-1,1)$ according to \cite{J}. The space $V_3 \subset F(M)$ of functions $f_{kl}$ such that $f_{kl} = -f_{lk}$ for all $k,l \in K$ and 
\begin{equation}\label{E:zero}
     \sum_{l = 1}^n f_{kl} = 0
\end{equation}
is irreducible. It is parametrized by the partition $(n-2,1,1)$ and has the dimension
\[
     \dim V_3 = \frac{n^2 - 3n + 2}{2} = \frac{(n-1)(n-2)}{2}.
\]
Finally, the space $V_4 \subset F(M)$ of functions $f_{kl}$ such that $f_{kl} = f_{lk},\  f_{kk} = 0$ for all $k,l \in K$, and satisfying condition (\ref{E:zero}) is also irreducible. It is parametrized by the partition $(n-2,2)$ and has the dimension
\[
     \dim V_4 = \frac{n^2 - 3n}{2} = \frac{n(n-3)}{2}.
\]
The space $E \subset F(M)$ of functions of the form $f_{kl} = a_k + b_l$ is $\mathfrak{S}_n$-invariant. Therefore, its orthogonal complement $E^\bot$ with respect to the Hermitian product $\langle\cdot,\cdot\rangle_u$ is also $\mathfrak{S}_n$-invariant. The restriction of the representation $R$ onto $E^\bot$ is multiplicity free. It is the following orthogonal sum of irreducible subspaces: 
\begin{equation}\label{E:direct}
 E^\bot = \left(V_1 \cap E^\bot\right) \oplus \left(V_2 \cap E^\bot\right) \oplus V_3 \oplus V_4. 
\end{equation}
Using the facts that the Berezin transform $I_u$ commutes with the representation $R$ and acts trivially on the space $E$, we will find the spectral decomposition of the operator $I_u$. It is scalar on the direct summands in (\ref{E:direct}).

Given a function $f = f_{kl} \in V_3$, we will calculate the matrices $(x_{kk'})$ and $(y_{kk'})$ of the operators $C_uf$ and $D_uf$, respectively. Observe that $f_{kk} = 0$ for all $k \in K$. Using formulas (\ref{E:cdsymb}) and (\ref{E:zero}), we get
\begin{eqnarray*}
  x_{kk'} = \sum_{l\in L} \left(\delta_{kl} + \frac{\theta - 1}{n}\right) f_{kl}\left(\delta_{k'l} + \frac{\bar\theta - 1}{n}\right) =\\
\sum_{l\in L} \left(\delta_{kl}f_{kl}\delta_{k'l} + \frac{\theta - 1}{n} f_{kl}\delta_{k'l} + \delta_{kl} f_{kl} \frac{\bar\theta - 1}{n} + \left| \frac{\theta - 1}{n} \right|^2 f_{kl}\right) = \\
f_{kk}\delta_{kk'} + \frac{\theta - 1}{n} f_{kk'} + f_{kk}\frac{\bar\theta - 1}{n} = \frac{\theta - 1}{n} f_{kk'}.
\end{eqnarray*}
A similar calculation shows that
\[
   y_{kk'} = \frac{\bar\theta - 1}{n}\, f_{k'k} = - \frac{\bar\theta - 1}{n}\, f_{kk'} = \bar \theta\, \frac{\theta - 1}{n}\, f_{kk'} = \bar\theta\, x_{kk'}.
\]
Thus $D_uf = \bar\theta\, C_uf$, whence $I_u f = \bar\theta f$.

Now assume that $f = f_{kl} \in V_4$. Denote by $(x_{kk'})$ and $(y_{kk'})$ the matrices of the operators $C_uf$ and $D_uf$, respectively. Then similar calculations show that
\[
     x_{kk'} = \frac{\theta - 1}{n}\, f_{kk'} \mbox{ and } y_{kk'} = \frac{\bar\theta - 1}{n}\, f_{kk'} = - \bar\theta\,x_{kk'},
\] 
whence $D_uf = -\bar\theta \, C_uf$ and therefore $I_u f = -\bar\theta\, f$.

Tedious but straightforward calculations render the eigenvalues of the operator $I_u$ on the eigenspaces $V_1\cap E^\bot$ and $V_2\cap E^\bot$. The spectral decomposition of the Berezin transform $I_u$ is summarized in the following table.

\medskip

\begin{tabular}{|l|c|c|c|c|c|}\hline
Eigenspace & $E$ & $V_1\cap E^\bot$ & $V_2\cap E^\bot$ & $V_3$ & $V_4$\\
\hline
Dimension & $2n-1$ & 1 & $n-1$ & $\frac{n^2 - 3n + 2}{2}$ & $\frac{n^2 - 3n}{2}$ \\
\hline
Eigenvalue & 1 & $-\theta \frac{\bar\theta + n -1}{\theta + n - 1}$ & $-\frac{\bar\theta + n - 1}{\theta + n - 1}$ & $\bar\theta$ & $-\bar\theta$\\ 
\hline
\end{tabular}

\medskip

Observe that the condition $\theta \neq \pm 1$ implies that the multiplicity of 1 in the spectrum of the Berezin transform $I_u$ is equal to $2n-1$.

\section{A mapping from the unitary to the doubly stochastic matrices}\label{S:geom}

Let $K$ and $L$ be two $n$-element index sets. Denote by $U$ the set of all unitary matrices $(u_{kl}), k\in K,l\in L$. It is a principal homogeneous space of the group of unitary matrices $(u_{kk'}), k,k'\in K$, and thus is a real $n^2$-dimensional compact manifold. Denote by $U_0$ the set of all unitary matrices from $U$ with nonzero entries. It is an open subset of $U$.

Let $B$ be the affine space of real matrices $(p_{kl}), k\in K,l\in L,$ satisfying the equations
\[
    \sum_{k \in K} p_{kl} =1 \mbox{ and } \sum_{l\in L} p_{kl} =1.
\]
Its dimension is $(n-1)^2$. The set $P$ of matrices from $B$ with nonnegative elements
is an $(n-1)^2$-dimensional convex polytope, the Birkhoff polytope (see \cite{Z}). Its elements are called doubly stochastic matrices. The Birkhoff polytope is the convex hull of the permutation matrices $(\delta_{\sigma(k)l})$, where $\sigma: K \to L$ is a bijection. Its interior consists of the doubly stochastic matrices with nonzero entries.

Two unitary matrices $u = (u_{kl})$ and $\tilde u = (\tilde u_{kl})$ from $U$ are called equivalent if there exist unitary diagonal matrices $\varkappa = (\varkappa_k \delta_{kk'}), k,k' \in K,$ and $\lambda = (\lambda_l \delta_{ll'}), l,l'\in L,$ such that $\tilde u = \varkappa u \lambda$, i.e., $\tilde u_{kl} = \varkappa_k u_{kl} \lambda_l$. We will write $u \sim \tilde u$ for equivalent matrices $u$ and $\tilde u$.

Denote by $\mu: U \to B$ the mapping that maps a unitary matrix $(u_{kl})$ to the doubly stochastic matrix $(|u_{kl}|^2)$. It is known that the image $\mu(U)$ is a proper subset of the Birkhoff polytope $P$ (see \cite{TZ}). 

The mapping $\mu$ maps equivalent matrices to the same doubly stochastic matrix. Thus it factors through the singular qoutient space $U/\sim$. The quotient space $U_0/\sim$ is an $(n-1)^2$-dimensional manifold. These observations lead to the natural question of whether the mapping $\mu$ is a submersion. It is obvious that $\mu$ is not a submersion on the complement $U \backslash U_0$ since it is mapped to the boundary of the Birkhoff polytope. In the rest of this paper we will show that for each $n$ there exists a unitary $(n\times n)$-matrix $u$ such that $\mu$ is a submersion at $u$. Since $\mu$ is an algebraic mapping of real affine manifolds, this will imply that $\mu$ is a submersion for almost all elements of $U$.

Let $u = (u_{kl})$ be an arbitrary unitary matrix in $U$ and $\U: F(L) \to F(K)$ be the corresponding unitary isomorphism. We want to give different descriptions of the tangent space $T_u U$. Assume that $u(t) = u_{kl}(t) \in U$ is a family of unitary matrices smoothly depending on a small real parameter $t$ and such that $u(0) = u$. Then $u'(0) = (u'_{kl}(0))$ is a tangent vector in $T_u U$. Here we use that $U$ is a real submanifold in the affine space of complex-valued matrices $(x_{kl})$ with $k\in K$ and $l\in L$. 
This way one can obtain all tangent vectors in $T_u U$. Each matrix $u(t)$ determines a unitary isomorphism $\U(t): F(L) \to F(K)$. Since $\U(t)(\U(t))^* = 1$ and $\U(0) = \U$, we get that
\[
           \U'(0)\U^* + \U(\U'(0))^* = 0.
\]
The operator
\begin{equation}\label{E:X}
     X = \U'(0)\U^* = - \U(\U'(0))^*
\end{equation}
in $F(K)$ corresponding to the tangent vector $u'(0)$ is skew-Hermitian. This correspondence identifies the tangent space $T_u U$ with the space of skew-Hermitian operators in $F(K)$. 

Now assume that $u\in U_0$.  Corollary \ref{C:skew} allows to identify the space of skew-Hermitian operators in $F(K)$ (and thus the tangent space $T_u U_0$) with the space of $C$-symbols $f \in F(M)$ such that $f = - I_u\bar f$ and also with the space of $D$-symbols $g \in F(M)$ such that $I_ug = - \bar g$.

Formula (\ref{E:X}) implies that the matrix $(x_{kk'}), k,k'\in K,$ of the operator $X$ satisfies the following equations:
\begin{equation}\label{E:xkk}
    x_{kk'} = \sum_{l\in L} u'_{kl}(0) \bar u_{k'l} = - \sum_{l\in L} u_{kl} \bar u'_{k'l}(0). 
\end{equation}
Assume that $f=f_{kl}$ and $g = g_{kl}$ are the $C$- and $D$-symbol of the operator $X$, respectively, that is,
\[
     X = C_uf = D_ug. 
\]
It follows from formulas (\ref{E:cdsymb}) and (\ref{E:xkk}) that
\begin{equation}\label{E:fandgforX}
     f_{kl} = \frac{u'_{kl}(0)}{u_{kl}} \mbox{ and } g_{kl} = -\frac{\bar u'_{kl}(0)}{\bar u_{kl}}.
\end{equation}
In particular, $f = -\bar g$. Formulas (\ref{E:X}) and (\ref{E:fandgforX}) mean that the
$C$- and $D$-symbols are related to the perturbations of the matrix elements of $u$ while the operators corresponding to these symbols are related to the perturbations of the isomorphism $\U$.

The mapping $\mu$ maps the matrix $u(t) = (u_{kl}(t))$ to the doubly stochastic matrix $p(t) = (p_{kl}(t))$ such that
\[
     p_{kl}(t) = |u_{kl}(t)|^2 = u_{kl}(t)\bar u_{kl}(t).
\]
We set $p = p(0)$. The tangent mapping $\mu_{*u}$ maps the tangent vector $u'(0) \in T_{u}U_0$ to the tangent vector $p'(0) \in T_{p}P_0$. It follows from formulas (\ref{E:fandgforX}) that
\[
    p'_{kl}(0) = u'_{kl}(0)\bar u_{kl} + u_{kl}\bar u'_{kl}(0) = \left(f_{kl} - g_{kl}\right)|u_{kl}|^2.
\]
Thus the tangent vector $u'(0)\in T_u U_0$ lies in the kernel of the tangent mapping $\mu_{*u}$ if and only if the $C$- and $D$-symbols of the corresponding skew-Hermitian operator $X$ coincide, $f = g$, or, equivalently, $I_u g = g$, that is, $g$ is an eigenvector of the operator $I_u$ with the eigenvalue 1.
Since the operator $X$ is skew-Hermitian, according to Corollary \ref{C:skew}, $f = -\bar g$ and therefore the condition that $f = g$ implies that $g = -\bar g$, i.e., that $g$ is purely imaginary. We conclude that the kernel of the tangent mapping $\mu_{*u}$ can be identified with the space of purely imaginary eigenfunctions of the Berezin transform $I_u$ with the eigenvalue 1. Now, Lemma \ref{L:eigen} implies the following

\begin{theorem}\label{T:ker}
  Given a unitary matrix $u \in U_0$, the dimension of the kernel of the tangent mapping $\mu_{*u}$ is equal to the multiplicity of 1 in the spectrum of the Berezin transform $I_u$. 
\end{theorem}

Theorem \ref{T:ker} and Example 2 from Section \ref{S:examples} show that the mapping $\mu$ is a submersion at the matrix $u = (u_{kl})$ given by (\ref{E:example2}). Thus, $\mu$ is a submersion almost everywhere on $U$.

\end{document}